\documentclass{amsart}
\usepackage{amsfonts,amssymb,amsmath,mathtools,mathrsfs,bm,stmaryrd,amsthm, enumerate}
\usepackage[utf8]{inputenc}
\usepackage[hidelinks]{hyperref}
\usepackage[T1]{fontenc}
\usepackage[english]{babel}
\usepackage[capitalise]{cleveref}

\usepackage[parfill]{parskip}
\usepackage{wrapfig,tikz, tikz-cd}
\usetikzlibrary{arrows}

\newtheorem{theorem}{Theorem}
\newtheorem{theorema}{Theorem}

\newtheorem{lemma}{Lemma}

\newtheorem{corollary}{Corollary}

\theoremstyle{definition}
\newtheorem{definition}{Definition}
\newtheorem{remark}{Remark}

\newtheorem{question}{Question}

\newcommand{\Hom}{{\operatorname{Hom}}}

\newcommand{\HH}{{\operatorname{HH}}}

\newcommand{\Ext}{{\operatorname{Ext}}}

\DeclareMathOperator{\Aut}{Aut}
\DeclareMathOperator{\Der}{Der}
\DeclareMathOperator{\iDer}{Der_{\rm int}}

\DeclareMathOperator{\obs}{obs}

\def\sl{\mathfrak{sl}}
\def\e{\mathfrak{e}}
\def\f{\mathfrak{f}}
\def\h{\mathfrak{h}}

\title{On the Lie algebra structure of integrable derivations}

\author{Benjamin Briggs}
\address{
Mathematical Sciences Research Institute, Berkeley, California}
\email{briggs@math.utah.edu}

\author{Lleonard Rubio  y Degrassi}
\address{Dipartimento di Informatica - Settore di Matematica, Universit\`{a} degli Studi di Verona, Strada le Grazie 15 - Ca’ Vignal, I-37134 Verona, Italy}
\email{lleonard.rubioydegrassi@univr.it}

\thanks{{\it Funding.} During this work the first author was hosted by the Mathematical Sciences Research Institute in Berkeley, California,
  supported by the National Science Foundation under Grant No.\ 1928930. 
The second author has been partially  supported by EPSRC grant EP/L01078X/1 and the project PRIN 2017 - Categories, Algebras: Ring-Theoretical and Homological Approaches. He also acknowledges support from the project {\it REDCOM:  Reducing complexity in algebra, logic, combinatorics}, financed by the programme {\it Ricerca Scientifica di Eccellenza 2018} of the Fondazione Cariverona.\\}
\begin{document}

\maketitle
\begin{abstract}
Building on work of Gerstenhaber, we show that the space of integrable derivations on an Artin algebra $A$ forms a Lie algebra, and  a restricted Lie algebra if $A$ contains a field of characteristic $p$. We deduce that the space of integrable classes in $\HH^1(A)$ forms a (restricted) Lie algebra that is invariant under derived equivalences, and under stable equivalences of Morita type between self-injective algebras. We also provide negative answers to questions about integrable derivations posed by Linckelmann and by Farkas, Geiss and  Marcos.
\end{abstract}
\section{Introduction}

For any algebra $A$ over a commutative ring $k$, we consider the Lie algebra of $k$-linear derivations on $A$. The subspace of integrable derivations was introduced by Hasse and Schmidt in \cite{HS}, and has since been important in geometry and commutative algebra, especially in regards to deformations, jet spaces, and automorphism groups \cite{Ma,Mo,Vo}. 

More recently, integrable derivations have been  used as source of invariants in representation theory \cite{FGM,ML}. The first Hochschild cohomology Lie algebra $\HH^1(A)$, consisting of all derivations modulo inner derivations, is a critically important invariant, and the subspace $\HH^1_{\rm int}(A)$ spanned by integrable derivations is the main object of interest in this work. This class of derivations is known to have good invariance properties: Farkas, Geiss and Marcos prove that $\HH^1_{\rm int}(A)$ is an invariant under Morita equivalences \cite{FGM}, and Linckelmann proves, for self-injective algebras, that  $\HH^1_{\rm int}(A)$ is an invariant under stable equivalences of Morita type \cite{ML}.

The first part of this work builds heavily on Gerstenhaber's work on integrable derivations \cite{G}, which seems not to be well-known. Following this work, we prove that for Artin algebras, the integrable derivations form a (restricted) Lie algebra.

\begin{theorema}[See Corollary \ref{cor_Lie}]
\label{theo_lie}
If $A$ is an Artin algebra over a commutative Artinian ring $k$, then $\HH^1_{\rm int}(A)$  is a Lie subalgebra of $\HH^1(A)$. If moreover $A$ contains a field of characteristic $p$, then $\HH^1_{\rm int}(A)$  is a restricted Lie subalgebra  of $\HH^1(A)$. 
\end{theorema}

We also show that this (restricted) Lie algebra is invariant under derived equivalences and   stable equivalences of Morita type, extending the work of Farkas, Geiss and Marcos \cite{FGM} and Linckelmann \cite{ML}.

\begin{theorema}[See Theorem \ref{invarianceint}]
\label{theoinvarianceint}
Let $A$ and $B$ two finite dimensional split algebras over a field $k$. Assume either that $A$ and $B$ are derived equivalent, or that $A$ and $B$ are self-injective and stably equivalent of Morita type. Then $\HH_{\rm int}^1(A)\cong \HH_{\rm int}^1(B)$ as Lie algebras, and this is an isomorphism of restricted Lie algebras if $k$ is of positive characteristic.
\end{theorema}

In the second part of this work we answer several  questions  posed in \cite{FGM} and \cite{L} about integrable derivations on group algebras. We  show by example that the Lie algebra of integrable derivations is not always solvable for every block of a finite group, proving a negative answer to {\cite[Question 8.2]{L}}.

\begin{theorema}[See Theorem \ref{notsolvable}]
\label{theonotsolvable}
Let $k$ be a field of characteristic $p\geq 3$, let $C_p$ the cyclic group of order $p$ and let $A=k(C_p\times C_p)$.
Then $\HH^1_{\rm int}(A)$ is not solvable. 
\end{theorema}

By \cite[Theorem 2.2]{FGM} the group algebra of a finite $p$-group, over a field of characteristic $p$, always admits a non-integrable derivation. 
The authors of this work ask whether this is true of all finite groups with order divisible by $p$, and our next example shows that this is not the case.

\begin{theorema}[See Theorem \ref{Spint}]
\label{theospint}
Let $k$ a field of characteristic $p\geq 3$  and let $kS_p$ be the group algebra of the symmetric
group on $p$ letters. Then $\mathrm{dim}_k(\HH^1(kS_p))=1$ and  $\HH^1(kS_p)= \HH^1_{\rm int}(kS_p)$.
\end{theorema}

Along the way, in Theorem \ref{dim} we give a formula for the dimension of $\HH^1(kS_n)$ for any $n$. The same formula has also been obtained independently in the recent work \cite{BKL}.  The first part of Theorem \ref{Spint}, that  $\mathrm{dim}_k(\HH^1(kS_p))=1$, is an immediate consequence.

\subsection*{Outline} The sections of this paper can be read independently. 
In Section \ref{sec_int_Ger} we survey some of  work of Gerstenhaber on integrable derivations and note some consequences. Here we prove Theorems \ref{theo_lie} and \ref{theoinvarianceint}. The main result of Section \ref{cexamplesolv}  is Theorem \ref{theonotsolvable} which provides the first counterexample. In Section \ref{HH1symgroup} we give a formula for the dimension of $\HH^1(kS_n)$. The main result in Section \ref{cexampleexist}  is Theorem \ref{theospint} which provides the second counterexample. The appendix contains a dictionary explaining the more general terminology used by Gerstenhaber in \cite{G}, and how it relates to the setting considered here.

\section{Integrable derivations}\label{sec_int_Ger}
Let $A$ be an algebra over a commutative ring $k$, and let $\Der(A)$ denote the space of $k$-linear derivations on $A$. 
Gerstenhaber investigated integral derivations in \cite{G}, but since his work was written in substantial generality, in the language of ``composition complexes", many of its results are not well-known. In this section we present some of the results of \cite{G} in more familiar terms (but in particular Lemma \ref{local_lem} and its consequences are new to this work). The main result of this section is that if $A$ is an Artin algebra, the class of integrable derivations forms a Lie subalgebra of $\Der(A)$.

In order to define integrable derivations we consider the $k$-algebras $A[t]/(t^n)$  
and their limit $A\llbracket t\rrbracket$. 
We denote by $\Aut_1(A[t]/(t^n))$ the group of $k[t]/(t^n)$-algebra automorphisms $\alpha$ which yield the identity modulo $t$. Any such automorphism can be expanded
\[
\alpha = 1 + \alpha_1t+\alpha_{2}t^{2}+\cdots + \alpha_{n-1}t^{n-1}
\]
for some $k$-linear maps $\alpha_i\colon A\to A$. The first of these, $\alpha_1$, is a derivation on $A$, and in general the maps satisfy $\alpha_i(xy)=x\alpha_i(y) + \alpha_1(x)\alpha_{i-1}(y)+\dots + \alpha_i(x)y$ for all $i$. A sequence $\alpha_1,...,\alpha_{n-1}$ of linear endomorphsims of $A$ satisfying these identities is called a Hasse-Schmidt derivation of order $n$. One similarly interprets elements $\alpha\in \Aut_1(A\llbracket t\rrbracket)$ as infinite sequences $\alpha_1,\alpha_2,...$ of linear endomorphsims of $A$, and these are called Hasse-Schmidt derivations of infinite order. These were studied in \cite{HS} under the name higher derivation.

Extending this slightly, we set $\Aut_m(A[t]/(t^n))$ to be the group of $k[t]/(t^n)$-algebra automorphisms $\alpha$ which yield the identity modulo $t^m$. Any such automorphism can be expanded
\[
\alpha = 1 + \alpha_mt^m+\alpha_{m+1}t^{m+1}+\cdots + \alpha_{n-1}t^{n-1}
\]
for some $\alpha_i\colon A\to A$. The first nonvanishing coefficient $\alpha_m$ is always a derivation on $A$. The same goes for the power series algebra $A\llbracket t\rrbracket$ and the corresponding automorphisms in $\Aut_m(A\llbracket t\rrbracket)$.

\begin{definition}[\cite{HS,G,R}]
A  $k$-linear derivation  $D$ on $A$ will be called $[m,n)$-{integrable} if there is an automorphism $\alpha\in \Aut_m(A[t]/(t^n))$ such that $D=\alpha_m$. 
We say that $D$ is {$[m,\infty)$-integrable} if it is {$[m,n)$-integrable} for all $n$. And we say that $D$ is {$[m,\infty]$-integrable} if there is an automorphism $\alpha\in \Aut_m(A\llbracket t\rrbracket)$ such that $D=\alpha_m$.  
We will write
\[
\Der_{[m,n)}(A) =\big\{ \ k\text{-linear } [m,n)\text{-integrable derivations on } A\ \big\},
\]
and we will use similar notation for $[m,\infty)$- and $[m,\infty]$-integrable derivations.

The derivations that are $[1,\infty]$-integrable are simply known as integrable, and we will also use the notation $\iDer(A)=\Der_{[1,\infty]}(A)$.
\end{definition}

\begin{remark}
The $[m,n)$-{integrable} derivations are closed under addition and subtraction---that is, $\Der_{[m,n)}(A)$ is an additive subgroup of $\Der(A)$. If $\alpha,\alpha'\in \Aut_m(A[t]/(t^n))$ then the automorphisms
\[
\alpha\alpha' = 1 + (\alpha_m+\alpha'_m)t^m+\cdots \quad\text{and}\quad \alpha^{-1}= 1-\alpha_mt^m+\cdots
\]
are witness to fact that $\alpha_m+\alpha'_m$ and $-\alpha_m$ are  $[m,n)$-{integrable}. The same goes for for $[m,\infty)$- and $[m,\infty]$-integrable derivations.

In the case $m=1$, $\Der_{[1,n)}(A)$ is moreover a submodule of $\Der(A)$ over ${\rm Z}(A)$---the centre of $A$. This can be seen from the automorphism
\[
1+z\alpha_1t+z^2\alpha_2t^2+\cdots
\]
that exists for any $\alpha \in \Aut_1(A[t]/(t^n))$ and $z\in {\rm Z}(A)$.
\end{remark}                                   

\begin{remark}
All inner derivations are integrable. Indeed if $a\in A$, then $1+at$ is a unit in $A\llbracket t\rrbracket$ and the automorphism ${\rm ad}(1+at) = 1 + [a,-]t+ \cdots $ shows that $[a,-]$ is integrable.
\end{remark}

\begin{definition}
The first Hochshild cohomology of $A$ over $k$ is the quotient $\HH^1(A)$ of $\Der(A)$ by the space of inner derivations. We denote by $\HH^1_{\rm int}(A)$ the image of $\iDer(A)$ in $\HH^1(A)$. By the previous two remarks, a class in Hochschild cohomology is integrable if and only if all or any one of its representatives is integrable.
\end{definition}

\begin{remark} \label{rem_bracket_restr}
If $D$ and $D'$ are derivations that are $[m,n)$-integrable and $[m',n)$-integrable respectively, then the commutator $[D,D']$ is an $[m+m',n)$-integrable derivation. Indeed, a computation shows that
\[
\alpha \alpha'\alpha^{-1}\alpha'^{-1} = 1 + (\alpha_m\alpha'_{m'}-\alpha'_{m'}\alpha_{m})t^{m+m'}+\cdots 
\]
for $\alpha \in \Aut_m(A[t]/(t^n))$ and $\alpha' \in \Aut_{m'}(A[t]/(t^{n}))$.

If $A$ contains a field of characteristic $p$, then the $p$th power of any derivation is again a derivation, and together with the commutator bracket this makes $\Der(A)$ into a restricted Lie algebra. If $D$ is an $[m,n)$-integrable derivation then $D^p$ is a $[pm,n)$-integrable derivation. Indeed, this time one computes
\[
\alpha^p = 1 + \alpha_m^pt^{pm}+\cdots 
\]
for $\alpha \in \Aut_m(A[t]/(t^n))$. This structure is studied in \cite{R}.

This remark shows that $[m,n)$-integrable derivations arise even if one is interested only in $[1,n)$-integrable derivations.
\end{remark}

Gerstenhaber works locally in \cite{G}, assuming that $k$ is an algebra over $\mathbb{Z}_p$ (the integers localised at $p$) for some prime $p$. Using the next lemma, in which we write $k_p=k\otimes_{\mathbb{Z}}\mathbb{Z}_p$ and $A_p=A\otimes_{\mathbb{Z}}\mathbb{Z}_p$, we can readily reduce to this case. Readers interested in algebras over local rings or fields can disregard the lemma.

\begin{lemma}\label{local_lem}
Let $A$ be a Noether algebra over a commutative Noetherian ring $k$. A $k$-linear derivation on $A$ is {$[m,n)$-integrable} if and only if for all primes $p$, the induced $k_p$-linear derivation on $A_p$ is {$[m,n)$-integrable}. 
\end{lemma}

\begin{proof}
The forward implication is clear. Conversely, assume that for each prime $p$ there is an automorphism $\alpha= 1+ \alpha_mt +\cdots + \alpha_{n-1}t^{n-1}$ in $\Aut_m(A_p[t]/(t^n))$ such that $\alpha_m=D$. 

Consider the element $(\alpha_i)\in \bigoplus_{i=m}^{n-1}\Hom_{k_p}(A_p,A_p)$. Since $k$ is Noetherian and $A$ is finitely generated as a $k$-module, $\bigoplus_{i=m}^{n-1}\Hom_{k_p}(A_p,A_p)\cong \big(\bigoplus_{i=m}^{n-1}\Hom_{k}(A,A)\big)_p$, and therefore there is a sequence $(\beta_i)\in \bigoplus_{i=m}^{n-1}\Hom_{k}(A,A)$ and an integer $u$ coprime to $p$ such that $(\alpha_i)=(\frac{\beta_i}{u})$ in $\bigoplus_{i=m}^{n-1}\Hom_{k_p}(A_p,A_p)$, and we can assume that $\beta_m= uD$. There is then an equality 
\[
1+ u^m\alpha_mt +\cdots + u^{n-1}\alpha_{n-1}t^{n-1} =  1+ u^{m-1}\beta_mt^m +\cdots + u^{n-2}\beta_{n-1}t^{n-1}
\]
in $\Aut_m(A_p[t]/(t^n))$. In particular, for each $m\leqslant i<n$ the Hasse-Schmidt identity
\[
u^{i-1}\beta_i(xy)-u^{i-1}x\beta_i(y) - u^{i-2}\beta_m(x)\beta_{i-m}(y)-\dots -u^{i-1} \beta_i(x)y =0  
\]
holds, when interpreted in $\Hom_{k_p}(A_p\otimes_{k_p}A_p,A_p)$. Since $\Hom_{k_p}(A_p\otimes_{k_p}A_p,A_p)\cong \Hom_{k}(A\otimes_{k}A,A)_p$ we may find an integer $v_i$ coprime to $p$ such that 
\[
v_i\big[u^{i-1}\beta_i(xy)-u^{i-1}x\beta_i(y) - u^{i-2}\beta_m(x)\beta_{i-m}(y)-\dots -u^{i-1} \beta_i(x)y\big] =0  
\]
holds in $\Hom_{k}(A\otimes_{k}A,A)$. Now set $v=v_m\cdots v_{n-1}$. It follows that the sequence 
\[
 (v^mu^{m-1}\beta_m, \ldots, v^{n-1} u^{n-2}\beta_{n-1})
\]
satisfies the Hasse-Schmidt identities in $\Hom_{k}(A\otimes_{k}A,A)$ for all $m\leqslant i<n$, and therefore defines an element of $\Aut_m(A[t]/(t^n))$. Now set
\[
\gamma_p = 1+v^mu^{m-1}\beta_mt^m+ \cdots+ v^{n-1} u^{n-2}\beta_{n-1}t^{n-1} \quad\text{and}\quad w_p=v^mu^m
\]
(up until now our notation has not indicated the dependence on $p$). Note that $w_pD=v^mu^{m-1}\beta_m$ by construction.

The ideal $(w_p ~:~ p\text{ is a prime}) \subseteq \mathbb{Z}$ contains an element coprime to every prime, and  is consequently the unit ideal. This means there are primes $p_1,...,p_i$ and integers $a_1,...,a_i$ such that $a_1w_{p_1}+\cdots +a_iw_{p_i}=1$. The automorphism 
\[
\gamma=\gamma_{p_1}^{a_1}\cdots \gamma_{p_i}^{a_i}\in \Aut_m(A[t]/(t^n))
\]
has as its $t^m$ coefficient $a_1w_{p_1}D+\cdots+ a_iw_{p_i}D=D$. Therefore $\gamma$ shows that $D$ is $[m,n)$-integrable.
\end{proof}

Certainly all $[m,\infty]$-integrable derivations are $[m,\infty)$-integrable. It seems to be open in general  whether the inclusion $\Der_{[m,\infty)}(A)\subseteq \Der_{[m,\infty]}(A)$ can be strict. However, the next result, which is essentially due to Gerstenhaber, shows that equality holds for Artin algebras.

\begin{lemma}\label{lem_infty}
Let $A$ be an Artin algebra over a commutative Artinian ring $k$. A $k$-linear derivation on $A$ is {$[m,\infty)$-integrable} if and only if it is  {$[m,\infty]$-integrable}.
\end{lemma}

\begin{proof}
By Lemma \ref{local_lem} we may assume that $k$ is an algebra over $\mathbb{Z}_p$, so that the results of \cite{G} apply. Since $\Der(A)$ is an Artinian $k$-module, the sequence of submodules $\Der_{[m,n)}(A)$ must eventually stablise, and so there is an $n$ such that  $\Der_{[m,n)}(A)=\Der_{[m,\infty)}(A)$. By {\cite[Theorem 5]{G}} this implies that $\Der_{[m,\infty)}(A)=\Der_{[m,\infty]}(A)$.
\end{proof}

It is easy to show that a $[1,\infty]$-integrable derivation is $[m,\infty]$-integrable for every $m$ (see \cite[Theorem 3.6.6]{R2}). The next result, due to Gerstenhaber, shows that the converse is also true.

\begin{theorem}\label{th_m_int}
Let $A$ be an Artin algebra over a commutative Artinian ring $k$. A $k$-linear derivation on $A$ is {$[1,\infty]$-integrable} if and only if it is  {$[m,\infty]$-integrable}. 
\end{theorem}

\begin{proof}
By Lemma \ref{local_lem} we may assume that $k$ is an algebra over $\mathbb{Z}_p$. Let $D$ be an {$[m,\infty]$-integrable} derivation on $A$. By Lemma \ref{lem_infty} there is an $n$ such that $\Der_{[1,n)}(A)=\Der_{[1,\infty]}(A)$, so it suffices to show that $D$ is $[1,n)$-integrable. There is an automorphism
\[
\alpha  = 1+Dt^m+\alpha_{m+1}t^{m+1}+\cdots +\alpha_{mn}t^{mn}
\]
in $\Aut_m(A[t]/(t^{mn+1}))$, and applying \cite[Theorem 3]{G} to this yields an automorphism 
\[
\alpha'  = 1+Dt^m+\alpha_{2m}'t^{2m}+\cdots +\alpha_{mn}'t^{mn}
\]
involving only powers of $t^m$. Replacing $t^m$ with $t$ finishes the proof.
\end{proof}

At this point we can deduce that $\iDer(A)=\Der_{[1,\infty]}(A)$ forms a Lie algebra; this was originally proven by Gerstenhaber   \cite[Corollary 1]{G}, assuming  that $k=k_p$ for some prime $p$.

\begin{corollary}\label{cor_Lie}
If $A$ is an Artin algebra over a commutative Artinian ring $k$, then $\iDer(A)$ is a Lie subalgebra of $\Der(A)$, and if moreover $A$ contains a field of characteristic $p$, then $\iDer(A)$ is a restricted Lie subalgebra of $\Der(A)$.  
By the same token, $\HH^1_{\rm int}(A)$ is a (restricted) Lie subalgebra of $\HH^1(A)$.
\end{corollary}

\begin{proof}
This follows from Remark \ref{rem_bracket_restr} and Theorem \ref{th_m_int}.
\end{proof}

The class of integrable derivations is known to have good invariance properties, and we may use Corollary \ref{cor_Lie} to upgrade these invariance results to statements about Lie algebras. The next result builds on the work of Rouquier, Huisgen-Zimmermann, Saor\'in, Keller, and Linckelmann. When $k$ has characteristic zero all derivations are integrable, so that $\HH_{\rm int}^1(A)=\HH^1(A)$, and the theorem is well-known in this case.

\begin{theorem}
\label{invarianceint}
Let $A$ and $B$ two finite dimensional split algebras over a field $k$. Assume either that $A$ and $B$ are derived equivalent, or that $A$ and $B$ are self-injective and stably equivalent of Morita type. Then $\HH_{\rm int}^1(A)\cong \HH_{\rm int}^1(B)$ as Lie algebras, and this is an isomorphism of restricted Lie algebras if $k$ is of positive characteristic.
\end{theorem}

\begin{proof}
Assume first that $A$ and $B$ are derived equivalent. The identity component $\mathrm{Out}(A)^\circ$ of the group scheme of outer automorphism group of $A$ is a derived invariant by \cite{HZS,Ruq}. Therefore the set 
\[
\mathrm{Out}_1(A\llbracket t\rrbracket)^\circ=\{ f\in \Hom_{k\text{-scheme}}({\rm Spec}( k\llbracket t\rrbracket),\mathrm{Out}(A)^\circ) ~:~ f(0)=1\}
\]
is a derived invariant as well, and we have  $\mathrm{Out}_1(A\llbracket t\rrbracket)^\circ\cong \mathrm{Out}_1(B\llbracket t\rrbracket)^\circ$. 

Since $\HH^1(A)\cong \Ext^1_{A^{\rm op}\otimes A}(A,A)$ the first Hochschild cohomology is also a derived invariant, and we have $\HH^1(A)\cong \HH^1(B)$.

Both isomorphisms above are realised by tensoring with a complex of $A$-$B$ bimodules \cite{Rick}, and it follows that the \emph{map} below is a derived invariant:
\[
\pi\colon \mathrm{Out}_1(A\llbracket t\rrbracket)^\circ \longrightarrow \HH^1(A) \quad \alpha = 1 + \alpha_1t+\alpha_2t^2+\cdots \mapsto \alpha_1.
\]
Therefore the image $\HH^1_{\rm int}(A)={\rm im}(\pi)$ is a derived invariant as well (c.f.~the proof of \cite[Theorem 5.1]{ML} for a similar argument). The fact that the isomorphism $\HH^1_{\rm int}(A)\cong \HH^1_{\rm int}(B)$ is one of Lie algebras now follows from Corollary \ref{cor_Lie} and \cite[Section 4]{KellerIII}. In positive characteristic this respects the $p$-power structure by \cite[(3.2)]{KellerI} combined with \cite[Theorem 2]{BR}.

For the case of self-injective algebras that are stably equivalent of Morita type, there is by \cite[Theorem 5.1]{ML} an isomorphism $\HH_{\rm int}^1(A)\cong \HH_{\rm int}^1(B)$ induced by a transfer map,   
and this is an isomorphism of restricted Lie algebras by Corollary \ref{cor_Lie} together with \cite[Theorem 1.1]{R} and \cite[Theorem 1]{BR}.
\end{proof}

\subsection{Obstructions to integrability}\label{sub_obs} The results of \cite{G} are proven using an obstruction theory for integrability, which we explain now.

For any automorphism $\alpha = 1+\alpha_1t^1+\cdots \alpha_{n-1}t^{n-1}$ in $\Aut_1(A[t]/(t^n))$ we define a $k$-linear map $\obs(\alpha)\colon A\otimes A\to A$ by the rule
\[
\widetilde{\obs}(\alpha)(x\otimes y)=\alpha_1(x)\alpha_{n-1}(y)+\dots + \alpha_{n-1}(x)\alpha(y).
\]
Viewed as a degree $2$ element in the Hochschild cochain complex $C^*(A)$, one checks that $\widetilde{\obs}(\alpha)$ is a cycle, and therefore defines a cohomology class
\[
\obs(\alpha)=[\widetilde{\obs}(\alpha)] \quad\text{in}\quad \HH^2(A).
\]
In order to extend $\alpha$ to an element of $\Aut_1(A[t]/(t^{n+1}))$, we must find a $k$-linear endomorphism $\alpha_n\colon A\to A$ satisfying the Hasse-Schmidt identity
\[
\alpha_n(xy)=x\alpha_n(y) + \alpha_1(x)\alpha_{n-1}(y)+\dots + \alpha_n(x)y. 
\]
This may be rearranged and formulated using the Hochschild cochain complex:
\[
\partial(\alpha_n) = \widetilde{\obs}(\alpha)\quad\text{in}\quad C^2(A).
\]
We obtain the statement of \cite[Proposition 5]{G}: an automorphism $\alpha\in \Aut_1(A[t]/(t^{n})) $ can be extended to $\Aut_1(A[t]/(t^{n+1}))$ if and only if $\obs(\alpha)=0$.

For any $[m,n)$-integrable derivation $D$, there is by definition an  automorphism $\alpha = 1+\alpha_mt^m+\cdots \alpha_{n-1}t^{n-1}$ in $\Aut_m(A[t]/(t^n))$ with $\alpha_m=D$. The above result shows that if $\obs(\alpha)=0$, then $D$ is $[m,n+1)$-integrable. We caution that $D$ being  $[m,n+1)$-integrable does not necessarily imply that $\obs(\alpha)=0$, only that \emph{some} automorphism in $\Aut_m(A[t]/(t^n))$ extending $1+Dt^m$ is unobstructed.

The \emph{obstruction order} of an automorphism $\alpha$ in $ \Aut_1(A[t]/(t^i))$ is $n$ if it admits an extension to an automorphism in $\Aut_1(A[t]/(t^n))$, but admits no extension to an automorphism in $\Aut_1(A[t]/(t^{n+1}))$. 

Gerstenhaber proves the following two facts about the above obstruction theory.

\begin{theorem}[{\cite[Theorem 1, Theorem 2, and Corollary 1]{G}}]\label{th_obs}
Let $A$ be an algebra over a commutative ring $k$. 
\begin{enumerate}
    \item\label{th_obs_1} If $\alpha,\alpha'\in \Aut_1(A[t]/(t^n))$ then  $\obs(\alpha\alpha')=\obs(\alpha)+\obs(\alpha')$ in $\HH^2(A)$.

    \item\label{th_obs_2} Assume $k=k_p$ for some prime $p$. If $D$ is a derivation then the obstruction order of  $1+Dt$ is (if finite) of the form $p^e$. Moreover, in this case the obstruction order of $1+Dt^m$ is $mp^e$.
\end{enumerate}
\end{theorem}

Because of (\ref{th_obs_2}), a derivation $D$ is said to have \emph{obstruction exponent} $e$ if $1+Dt$ has obstruction order $p^e$.

To give an example of how the obstruction theory is used we use it to give Gerstenhaber's beautiful proof of the analogue of Corollary \ref{cor_Lie} for finitely integrable derivations (which are especially connected with jet spaces).

\begin{corollary}[\cite{G}]
If $A$ is a Noether algebra over a commutative Noetherian ring $k$,
 then $\Der_{[1,n)}(A)$ is a Lie subalgebra of $\Der(A)$, and if moreover $A$ contains a field of characteristic $p$, then $\iDer(A)$ is a restricted Lie subalgebra of $\Der(A)$.
\end{corollary}

\begin{proof}
By Lemma \ref{local_lem} we may assume that $k=k_p$, so that Theorem \ref{th_obs} part (\ref{th_obs_2}) applies.

Taking $D,D'\in \Der_{[1,n)}(A)$, we may assume by Theorem \ref{th_obs} part (\ref{th_obs_2}) that $n=p^e$ is a power of $p$. Suppose that $D=\alpha_1$ and $D'=\alpha_1'$ for two $\alpha,\alpha'\in \Aut_1(A[t]/(t^{p^e}))$. As in Remark \ref{rem_bracket_restr} the automorphism $\alpha \alpha'\alpha^{-1}\alpha'^{-1} $ shows that  $[D,D']$ is $[2,p^e)$-integrable. However, by Theorem \ref{th_obs} part (\ref{th_obs_1}) we have $\obs(\alpha'\alpha^{-1}\alpha'^{-1})=0$, therefore $[D,D']$ is in fact $[2,p^e+1)$-integrable. By Theorem \ref{th_obs} part (\ref{th_obs_2}) it must therefore be $[2,2p^e)$-integrable, as $2p^{e-1}<p^e+1$. By \cite[Theorem 3]{G} this implies that $[D,D']$ is $[1,p^e)$-integrable.

A similar argument yields the second claim, since $\obs(\alpha^p)=p\obs(\alpha)=0$ using Theorem \ref{th_obs} part (\ref{th_obs_1}).
\end{proof}

\begin{remark}
Linckelmann considers a more general notion of integrable derivation in \cite{ML}, replacing $k\llbracket t \rrbracket$ with any discrete valuation ring. It would be interesting to develop the obstruction theory using this definition, and to see whether the results above extend to this context as well.
\end{remark}

\section{Counterexamples to solvability}
\label{cexamplesolv}

\subsubsection*{Notation} For the next three sections we denote by $S_n$ the the symmetric group on $n$ letters, by $A_n$ the alternating group and by $C_n$ the cyclic group of order $n$. When dealing with groups we use the superscript notation for the $n$-fold direct product and we denote by $H\rtimes N$ the semidirect product of $N$ and $H$. We denote by $k$ a field and by $k^{+}$ its additive group.

In this section we give a counter example concerning the solvability of integrable derivations.

\begin{question} [{\cite[Question 8.2]{L}}]
When is the Lie algebra $\HH_{\rm int}^1(A)$ solvable? 
\end{question}
In the same article Linckelmann suggests that, based on the examples,
$\HH^1_{\rm int}(A)$ should be a solvable Lie algebra if $A$ is a block of a
finite group algebra $kG$ over an algebraically closed field of prime characteristic. We provide a negative answer to this suggestion by considering the group algebra $kP$ of an elementary
abelian $p$-group $P$ of rank $2$, that is, $kP=k(C_p\times C_p)$ where $C_p$ is a cyclic group of order $p$. In this case the group algebra $kP$ coincides with its unique block.

\begin{theorem}
\label{notsolvable}
Let $k$ be a field of characteristic $p\geq 3$. Let $A=k(C_p\times C_p)$.
Then $\HH^1_{int}(A)$ is not solvable. 
\end{theorem}

\begin{proof}
Note that $A=k(C_p\times C_p)\cong k[x,y]/(x^p, y^p)$. Then  $\HH^1(A)$ is a Jacobson-Witt algebra \cite{J} and has a $k$-basis given by
the  derivations 
\[
\{f_{a,b}\ |\ 0 \leq a,b\leq p-1\}\cup \{g_{c,d}\ |\  0 \leq c,d\leq p-1\}
\] 
where $f_{a,b}(x)=x^ay^b$, $f_{a,b}(y)=0$ and $g_{c,d}(y)=x^cy^d$, $g_{c,d}(x)=0$. Note that $\HH^1_{\rm int}(A)$ has the same $k$-basis of $\HH^1(A)$ excluding  the derivations $f_{0,0}$ and $g_{0,0}$.
In order to prove that $f_{1,0}$ is an integrable derivation we construct the automorphism $\alpha=1+ f_{1,0}t$ $\in \mathrm{Aut}(A[[t]])$. It is easy to check that it preserves the relations. Using the same argument we can prove that $g_{0,1}$ is integrable as well. Then we use the fact that the space  of integrable derivations form a ${\rm Z}(A)$-module. Since $kP$ is a commutative algebra, we obtain that the rest of the derivations in the basis, aside from $f_{0,0}$ and $g_{0,0}$, are integrable. The derivations $f_{0,0}$ and $g_{0,0}$ do not preserve the Jacobson radical, hence they are not integrable (see Corollary 2.1 in \cite{FGM}).

Let $\f,\e,\h$ be  a basis of 
$\sl_2(k)$ satisfying
$[\e,\f] = \h$, $[\h,\f] = -2\f$, and $[\h,\e] = 2\e$. The derived subalgebra of $\HH^1_{\rm int}(A)$ 
contains
the Lie algebra $\sl_2(k)$ 
via the map sending $\f$ to 
$f_{0,1}(=[f_{0,1},f_{1,0}])$, $\h$ to $f_{1,0}-g_{0,1}(=[g_{1,0},f_{0,1}])$, and $\e$ to
$g_{1,0}(=[g_{1,0},g_{0,1}])$. The statement follows.
\end{proof}

By the same argument one shows:

\begin{corollary}
Let $k$ be a field of characteristic $p$. Let $P$ be an elementary abelian $p$-group of rank greater than $2$.  
Then $\HH^1_{\rm int}(kP)$ is not solvable. 
\end{corollary}

\section{The first Hochschild cohomology of the symmetric group}
\label{HH1symgroup}
In this section we give a formula for the dimension of $\HH^1(kS_n)$.  We start by recalling  some basic facts on the representation theory of the symmetric group.  

\begin{definition}
A \emph{partition} of a 
nonnegative integer $n$ is a 
decreasing sequence of positive integers $\lambda_1 > \lambda_2 
> \dots > \lambda_s 
>0$ and positive integers $e_1,\dots, e_s$ such that  such that 
$e_1\lambda_1+\dots +e_s\lambda_s=n$. We use the notation  $\lambda = (\lambda_1^{e_1},\dots,\lambda_s^{e_s})$, and we say that $\lambda_i$ the \emph{$i$th part} of 
$\lambda$, and $e_i$ is the \emph{multiplicity} of $\lambda_i$. We denote by 
$\mathcal{P}(n)$ the set of all partitions of $n$.

\end{definition}

We recall that the conjugacy classes of $S_n$ are in bijection with the 
partitions of $n$, with the conjugacy class of an element $x$ corresponding to the partition $\lambda$ determined by the cycle type of $x$. That is, if $x=c_{1,1} \dots c_{1,e_1}\dots c_{s,1} \dots c_{s,e_s}$ in disjoint cycle notation (including cycles of length one), with $c_{i,j}$ a cycle of length $\lambda_i$, then  $\lambda = (\lambda_1^{e_1},\dots,\lambda_s^{e_s})$.

We begin by computing $\HH^1(kS_n)$ using the centraliser decomposition of Hochschild cohomology.

\begin{theorem}
\label{dim}
Let $k$ a field of characteristic $p$ and let $S_n$ the symmetric group on $n$ 
letters. Then we have the following
decomposition:  
\[
\HH^1(kS_n)=\bigoplus_{\lambda\in \mathcal{P}(n)}\Hom \Big (\prod_{i=1}^{s}
(C_{\lambda_i}\times (S_{e_i}/A_{e_i})),k^{+} \Big ).
\]
\end{theorem}

\begin{proof}
Using the decomposition of $\mathrm{HH}^{1}(kS_n)$ into the direct sum of the first group cohomology of centraliser subgroups we have:
\[\HH^1(kS_n)=\bigoplus_{\lambda \in \mathcal{P}(n)} \mathrm{H}^1(C_{S_n}(x),k)=\bigoplus_{\lambda \in \mathcal{P}(n)} \Hom(C_{S_n}(x),k^{+})\]
Let $x$ be a representative element in each
conjugacy class of $kS_n$. The first step is to study the centraliser $C_{S_n}(x)$. As consequence of the fact that conjugation 
permutes cycles of the same length we have that
$C_{S_n}(x)=\prod_{i=1}^sC_{\lambda_i} \wr S_{e_i}$ where 
$\wr$ denotes the wreath product of $C_{\lambda_i}$ by 
$S_{e_i}$. In fact, there are two groups that sit
inside $C_{S_n}(x)$ and that generate $C_{S_n}(x)$. The first one is
$E:=S_{e_1}\times \dots \times S_{e_s}$ and  the second is
$\prod_{i=1}^s C^{e_i}_{\lambda_i}$. 
It is easy to check that
$C_{S_n}(x)=\prod_{i=1}^s(C^{e_i}_{\lambda_i} \rtimes S_{e_i})$ where $S_{e_i}$ acts on 
the direct product $C^{e_i}_{\lambda_i}$ by permutation.

The next step is to study the abelianisation of $C_{S_n}(x)$. 
Note that  the derived subgroup of $C_{S_n(x)}$ is given by 
\[\prod_{i=1}^s[C_{\lambda_i}\wr S_{e_i}, C_{\lambda_i}\wr 
S_{e_i}].\] 
In general, the derived subgroup of a semi-direct product $ N\rtimes H$
is equal to $([N,N][N,H])\rtimes[H,H]$. In our case $H=S_{e_i}$ and 
$N=C^{e_i}_{\lambda_i}$. So $[S_{e_i}, 
S_{e_i}]=A_{e_i}$ and $[N,N]=1$. It is easy to check that
$[C^{e_i}_{\lambda_i}, S_{e_i}]$ is isomorphic to
$C_{\lambda_i}^{e_{i}-1}$. Hence 
\[[C_{\lambda_i}\wr 
S_{e_i},C_{\lambda_i}\wr 
S_{e_i}]\cong C_{\lambda_i}^{e_{i}-1}\rtimes A_{e_i}\]
Consequently the abelianization of $C_{S_n}(x)$ is isomorphic to 
\[\prod_{i=1}^s (C_{\lambda_i}\times S_{e_i}/A_{e_i}).\] 
The statement follows.
\end{proof}

\begin{lemma}\label{lem_counting}
Let $p$ be a prime, and $n$ a nonnegative integer. The number of parts of length divisible by $p$ in all partitions of $n$, counted without multiplicity, is equal to the number of parts of length $p$ in all partitions of $n$, counted with multiplicity.
\end{lemma}

\begin{proof}
Using the notation $\lambda=(\lambda_1^{e_1},\dots,\lambda_s^{e_s})$ for a partition of $n$, we consider the set $\mathcal{S}_1$ of pairs $\{ (\lambda, e) \text{ with some }\lambda_i =p \text{ and } 1\leq e\leq e_i\}$, and the set $\mathcal{S}_2$ of pairs $\{ (\lambda, \lambda_i) \text{ with }p| \lambda_i \}$. We define a function $\mathcal{S}_1\to \mathcal{S}_2$ by the rule $(\lambda,e)\mapsto (\lambda',ep)$, where $\lambda'=(\lambda_1^{e_1},\dots,(ep)^{e'},\dots ,p^{e_i-e},\dots,\lambda_s^{e_s})$ and where $e'=e_j+1$ if $\lambda_j=ep$ was already a part of $\lambda$, or $e'=1$ if not (if it happens that $e=1$, then $\lambda=\lambda'$). This function is a bijection, and therefore $|\mathcal{S}_1|=|\mathcal{S}_2|$. Since $|\mathcal{S}_1|$ is the number of parts of length $p$ in all partitions of $n$, and $|\mathcal{S}_2|$ is the number of parts of length divisible by $p$ in all partitions of $n$ (without multiplicity), we are done.
\end{proof}

\begin{theorem}
\label{dimHH1Snot2}
If the characteristic of the field $k$ is different from $2$, then
\[
\HH^1(kS_n)\cong\bigoplus_{\lambda\in
\mathcal{P}(n)}\Hom \Big (\prod_{p | \lambda_i} C_{\lambda_i},k^{+}\Big)
\]

Therefore, $\dim_k(\HH^1(kS_n))$ is equal to the total number of parts, counted without multiplicity, divisible by $p$ for all partitions of $n$.  These numbers are equal to the number of $p$'s in all partitions of $n$ which are given by the generating series

\[
\sum_{n\geq 0} \dim_k(\HH^1(kS_n))t^n=\frac{t^{p}}{1-t^{p}}\prod_{n\geq 1}\frac{1}{(1-t^n)}.
\]
\end{theorem}

\begin{proof}
In Theorem \ref{dim}, the term $S_{e_i}/A_{e_i}$ will not contribute since the characteristic of the field is greater than $2$ and $S_{e_i}/A_{e_i}$ is either trivial or isomorphic to $C_2$. This yields the first statement. 

We also learn that $\dim_k(\HH^1(kS_n))$ is equal to the total number of parts divisible by $p$ in all partitions of $n$, counted without multiplicity. By Lemma \ref{lem_counting}, this the number of time $p$ occurs as a part in a partition of $n$. The generating series for this sequence can be found in \cite{Ri}. More precisely, we can associate to any sequence $(a_i)$ the function on partitions $L(\lambda):=\sum a_i \kappa_i$, where  $\kappa_i$ is the number of parts of size $i$ in $\lambda$; then 
 \cite[p185 Equation 23]{Ri} reads
\[
\sum_{n\geq0}t^n\sum_{\lambda\in \mathcal{P}(n)}L(\lambda)= \left(\sum_{n\geq 1}\frac{a_nt^{n}}{1-t^{n}}\right)\prod_{n\geq 1}\frac{1}{(1-t^n)}.
\]
If we take $a_p=1$ and $a_i=0$ for $i\neq p$ then we obtain the desired series.
\end{proof}

\begin{theorem}
If $k$ be a field of characteristic $2$ then 
\[
\sum_{n\geq 0} \dim_k(\HH^1(kS_n))t^n=\frac{2t^{2}}{1-t^{2}}\prod_{n\geq 1}\frac{1}{(1-t^n)}.
\]
\end{theorem}

\begin{proof}
By Theorem \ref{dim}
\[
\HH^1(kS_n)\cong\bigoplus_{\lambda\in
\mathcal{P}(n)}\Hom (\prod_{2| \lambda_i} C_{\lambda_i},k^{+})\oplus \Hom (\prod_{e_i\geq 2} C_{2},k^{+}).
\]
So the the computation of $\dim_k(\HH^1(kS_n))$ splits into two parts. For the first summand we count the number of parts in partitions on $n$ that are divisible by $2$; as in the proof of Theorem \ref{dimHH1Snot2} this is given by the generating series
\[
\frac{t^{2}}{1-t^{2}}\prod_{n\geq 1}\frac{1}{(1-t^n)}.
\]
For the second summand we must count the number of parts with multiplicity $2$ or more in all partitions of $n$. In the usual formula for the total number of partitions
\[
\sum_{n\geq0, \lambda\in\mathcal{P}(n)} t^n\ =\prod_{n\geq 1}\frac{1}{(1-t^n)}
\]
(cf.\ \cite{Ri}), the factor $1/{(1-t^i)}=(1+t^i+t^{2i}+\cdots)$ corresponds to parts of length $i$, with a term $t^{ie}$ contributing $1$ to the coefficient of $\lambda$ if $\lambda$ contains a part of length $i$ with multiplicity $e$. To modify this formula to count partitions with a chosen part of multiplicity $e\geq 2$, we simply replace this factor with $t^{2i}/{(1-t^i)}=(t^{2i}+t^{3i}+t^{4i}+\cdots)$. In total we get
\[
\sum_{i} \left(t^{2i}\prod_{n\geq 1}\frac{1}{(1-t^n)}\right) = \frac{t^{2}}{1-t^{2}}\prod_{n\geq 1}\frac{1}{(1-t^n)}.
\]
The statement of the Theorem follows.
\end{proof}

An element $x$ in a finite group is \emph{$p$-regular} if its order is coprime to $p$, and otherwise it is called \emph{$p$-singular}. 
In the case of $S_n$, the $p$-singular elements are those containing at least one cycle of length divisible by $p$. In other words, the corresponding partition  contains a part divisible by $p$. 
We write  $\mathcal{SP}(n)$ for the set of all  partitions of $n$ corresponding to conjugacy classes of $p$-singular elements.

\begin{corollary}
If $k$ is a field of characteristic $p$ then 
$\mathrm{dim}_k(\HH^1(kS_n))\geq |\mathcal{SP}(n)|$.
\end{corollary}

Finally, in the next section we will need the following fact.

\begin{corollary}
\label{S_p}
If $k$ is a field of characteristic $p>2$ then 
$\mathrm{dim}_k(\HH^1(kS_p))=1$.
\end{corollary}

\begin{proof}
By Theorem \ref{dim} we just need to count the the number of parts of length $p$ in all partitions of $p$, and there is clearly just one.
\end{proof}

\begin{remark}
Recently, the authors of \cite{BKL} have also computed $\dim_k(\HH^1(kS_n))$ in terms of generating functions.
\end{remark}

\section{Counterexamples to the existence of non-integrable derivations}
\label{cexampleexist}

In this section we answer a question considered by Farkas, Geiss and  Marcos.

\begin{question}[\cite{FGM}]
\label{question2}
Let $G$ be a finite group and let $k$ be a field such that $\mathrm{char}(k)$
divides the order of $G$, must $kG$ admit a non-integrable derivation?
\end{question}

Since all inner derivations are integrable 
 a necessary condition that should hold in order to state the previous question is the following: let $G$ be a finite group and assume the characteristic of the field $k$ divides the order of $G$. Then $\HH^1(kG)\neq 0$. This has been shown in \cite{FJLM} using the classification of finite simple groups. 
 
The authors state their question in terms of the automorphism group scheme, writing that \emph{It is tempting to conjecture that $kG$ does not have a smooth automorphism group scheme} \cite[below Theorem 2.2]{FGM}.
 Their question is equivalent to Question \ref{question2} by Theorem 1.2 in \cite{FGM}: the automorphism group scheme of a finite dimensional algebra $A$ is smooth if and only if every derivation of $A$ is integrable. 

In the following theorem we exhibit a family of counter-examples for any algebraically closed field of prime characteristic.

\begin{theorem}
\label{Spint}
Let $k$ a field of characteristic $p\geq 3$  and let $kS_p$ be the group algebra of the symmetric
group on $p$ letters. Then  $\HH^1(kS_p)$ has a $k$-basis given by a single integrable
 derivation.
\end{theorem}

The first part of  Theorem \ref{Spint} will follow from Corollary \ref{S_p}. To prove that the only outer derivation in $\HH^1(kS_p)$ is integrable, we will use the fact that the only non-semisimple block of $kS_p$ is derived equivalent to a symmetric Nakayama algebra.

We recall some basic results about blocks of symmetric groups.

A \emph {node} $(i,j)$ in the Young diagram $[\lambda]$ of $\lambda$ forms part of the \emph{rim} if $(i + 1, j + 1) \notin [\lambda]$. A \emph{$p$-hook} in $\lambda$ is a connected part of the rim of $[\lambda]$ consisting of exactly $p$ nodes, whose removal leaves the Young diagram  a partition. The \emph{$p$-core} of $\lambda$, usually denoted by $\gamma(\lambda)$, is the partition obtained by repeatedly removing all $p$-hooks from $\lambda$. The number of $p$-hooks we remove is  the \emph{$p$-weight} of $\lambda$, usually denoted by $w$. It is easy to note  that the $p$-core of a partition is well-defined, that is, is independent of the way in which we remove the $p$-hooks. The blocks  of  group algebras of symmetric groups are determined by  $p$-cores and  weights:

\begin{theorem}[Nakayama Conjecture]
The blocks of the symmetric group $S_n$ are labelled by pairs $(\gamma,w)$, where $\gamma$ is a $p$-core and $w$ is the associated $p$-weight such that $n = |\gamma| + pw$. Hence the Specht module $S^{\lambda}$ lies in the block labelled by $(\gamma, w)$ of $kS_n$ if and only if $\lambda$ has $p$-core  $\gamma$ and weight $w$.
\end{theorem}

Note that the statement above holds also for 
the simple modules, see the paragraph after 
Theorem 8.3.1 in \cite{Cra}. It is easy to see 
that blocks of weight $0$ are matrix algebras 
and blocks of weight $1$ have cyclic defect 
group.
 
Since we will be mainly interested in blocks with cyclic defect group, we recall some background on Brauer trees and Nakayama algebras. 

A Brauer graph consists of a finite undirected connected
graph, such that to each vertex we assign a cyclic ordering of the edges incident
to it, and an integer greater than or equal to one, called the multiplicity of
the vertex, see Definition 4.18.1 in \cite{Benson}. To each Brauer graph we can associate a finite dimensional algebra called Brauer graph algebra. A  Brauer tree is a Brauer graph which is a tree, and has at most one vertex with multiplicity greater than one. Such vertex is called the exceptional vertex and the multiplicity is called exceptional
multiplicity.  A Brauer star algebra  is a Brauer tree with a star-shape having $n$ edges and the exceptional vertex in the middle. 

Brauer tree algebras are important in modular representation theory because of the following result: A block $B$  with cyclic defect group of order $p^d $, having  $m$ simple modules, is a Brauer tree algebra with $m$ edges and with exceptional multiplicity $\frac{p^d-1}{m}$, see Theorem 6.5.5. in \cite{Benson}. For further background in modular representation theory of finite groups see also \cite{L1} and \cite{L2}.
 
A particular nice class of self-injective   algebras are self-injective Nakayama algebras. A basic algebra of a (connected) self-injective Nakayama algebra  is $kC_m/J_{m,n}$ where  where $C_m$ is the extended Dynkin quiver of type $\tilde{A}_m$ and  $J_{m,n}$ is the ideal in the path algebra $kC_m$ generated by the composition of $n+1$ consecutive arrows. Note also that every basic self-injective Nakayama algebra is a monomial algebra. In addition, the algebra $N^n_m$ is symmetric if and only if $m$ divides $n$, see \cite{Skow} for example.  The symmetric Nakayama algebra $N^{em}_{e}$ having $e$  simple
modules and Loewy length $em+1$ is  a Brauer
tree algebra with respect to the Brauer star with $e$ vertices and with exceptional
multiplicity $em$. It is worth noting that not every Brauer tree algebra is isomorphic is a symmetric Nakayama algebra, however, a result due to Rickard \cite{Rick} shows that every Brauer tree algebra is derived equivalent to a symmetric Nakayama algebra, see also \cite[Theorem 6.10.1]{Zim}.

\begin{theorem}
\label{Naktree}
Let $k$ be a field and let A be a Brauer tree algebra associated to a Brauer tree with $e$ edges and with
exceptional multiplicity $m$. Then $A$ is derived equivalent to $N^{me}_{e}$ where $N^{me}_{e}$ is the symmetric Nakayama algebra having $e$ vertices and admissible ideal $J_{e,me}$.
\end{theorem}

We have all the ingredients to prove  Theorem \ref{Spint}: 

\begin{proof}[Proof of Theorem \ref{Spint}]
The principal block of $kS_p$, denoted by $B_0$, has cyclic defect $C_p$ and the number of simple modules  of $B_0$ is $p-1$. This follows by Nakayama Conjecture since there are $p-1$ partitions having the same $p$-core of the partition representing the trivial module. The weight of $B_0$ is $1$ hence $B_0$ has cyclic defect $C_{p^d}$ for some $d$. In this case $B_0$  is a Brauer tree algebra for a Brauer tree with $p-1$ edges and exceptional multiplicity $1$, see after Example 5.1.4 in \cite{Cra}. Hence $B_0$ has cyclic defect $C_p$.
By Theorem \ref{Naktree} we have that $B_0$ is derived equivalent to a Nakayama algebra denoted by $N^{p-1}_{p-1}$. The Gabriel quiver associated with $N^{p-1}_{p-1}$ has a set of vertices given by $\{e_i\}_{i=1}^{p-1}$ and it has  $p-1$ arrows
$\{a_i\}_{i=1}^{p-1}$ such that $t(a_i)=s(a_{i+1})=e_{i+1}$ for $i\neq p-1$ and $t(a_{p-1})=s(a_1)$. The rest of 
the blocks of $kS_p$ are matrix algebras since they have weight 0. 
Therefore $\HH^1(kS_p)\cong\HH^1(N^{p-1}_{p-1})$ and this restricts to an isomorphism  $\HH^1_{\rm int}(kS_p)\cong\HH^1_{\rm int}(N^{p-1}_{p-1})$.

The first Betti number $\beta(Q)$ of the underying graph of $N^{p-1}_{p-1}$ is $1$. This is because the Gabriel quiver is connected, the number of edges is $p-1$, the number of vertices is $p-1$ and consequently $\beta(Q)=(p-1)-(p-1)+1=1$.
Since $N^{p-1}_{p-1}$ is monomial, by Theorem C in \cite{BR2} we have that the maximal toral rank is $1$ and it is easy to see that the  map $f$ sending $a_1$ to $a_1$
and sending any other arrow to zero is a diagonal outer derivation. From Corollary \ref{S_p}, we have $\mathrm{dim}_k(\HH^1(kS_p))=1$. We deduce that 
there are no other outer derivations. Recall that  $N^{p-1}_{p-1}$ is the bound quiver algebra $kC_{p-1}/J_{p-1,p-1}$ where $J_{p-1,p-1}$ is the ideal in the path algebra $kC_{p-1}$ generated by the composition of $p$ consecutive arrows. Let $\rho_1=(a_1\dots a_{p-1}a_1=0)$. Then any other relation that generates the ideal $J_{p-1,p-1}$ is given by the path that starts and ends at  $a_i$ for every $2\leq i\leq n$. We construct the automorphism 
$\alpha=1+ ft$ $\in \mathrm{Aut}(A[[t]])$.  This 
$k[[t]]$-automorphism preserves the relations. We check for $\rho_1$ since for the rest of generating relations the proof is similar.
We have 
\begin{equation*}
    \begin{split}
    \alpha(a_1\dots a_{p-1}a_1)&=\alpha(a_1)\dots\alpha(a_{p-1})\alpha(a_1)=(a_1+a_1t)a_2\dots a_{p-1}(a_1+a_1t)\\
    &=(a_1\dots a_{p-1}+a_1\dots a_{p-1}t)(a_1+a_1t)=0
\end{split}
\end{equation*}
Therefore $f$ is integrable and the statement follows.
\end{proof}

\begin{remark}
Note that in order to construct the previous counterexample we have 
considered a
Gabriel quiver without loops. In \cite{FGM} the authors consider $p$-groups, 
which are local algebras, hence all the arrows are loops. 
\end{remark}

\appendix
\section*{Appendix. Gerstenhaber's composition complexes}

In Section \ref{sec_int_Ger} we used results of Gerstenhaber \cite{G} to establish facts about the Lie algebra of integral derivations on an algebra. In this appendix we survey the more general definitions in \cite{G} and compare them with the context of Section \ref{sec_int_Ger}.

\begin{definition}[\cite{G}]\label{ccdef}
Let $k$ be a commutative ring. A composition complex $C$ over $k$ is a sequence $C^0,C^1,...$ of $k$-modules, and for each $m,n$ and $0\leq i\leq m-1$ a bilinear composition operation
\[
C^m\times C^n \to C^{m+n-1} \quad (f,g) \mapsto f\circ_ig,
\]
as well as for each $m,n$ a bilinear cup product operation 
\[
C^m\times C^n \to C^{m+n}\quad (f,g) \mapsto f\smile g,
\]
satisfying, for any $f\in C^n,g\in C^m$ and $h\in C^l$, the conditions
\[
(f\circ_i g)\circ_jh =
\begin{cases}
(f\circ_j h)\circ_{i+l-1} h& \text{if } 0\leq j \leq i-1\\
f\circ_i (g\circ_{j-i}h)  & \text{if } i\leq j \leq n-1,
\end{cases}
\]
and
\[
(f\smile g)\circ_jh =
\begin{cases}
(f\circ_i h)\smile g & \text{if } 0\leq i \leq m-1\\
f\smile (g\circ_{i-m}h)  & \text{if } m\leq i \leq m+n-1.
\end{cases}
\]
We further assume that the cup product of $C$ is associative, and that there is unit element $1\in C^1$ such that $1\circ_0 f = f\circ_i 1 = f$ for any $f\in C^m$ and $0\leq i\leq m-1$.
\end{definition}

The key example of a composition complex is the Hochschild cochain complex of a $k$-algebra $A$:
\[
C^n(A) = \Hom(A^{\otimes n}, A) \quad \text{with} \quad f\circ_ig = f \circ (1^{\otimes i} \otimes g \otimes 1^{\otimes m-i-1})
\]
for $f\in C^n$ and $g\in C^m$, and with the usual cup product 
\[
(f\smile g)(x_1\otimes \cdots x_{n+m}) = f(x_1\otimes \cdots x_n)g(x_{n+1}\otimes \cdots x_{n+m}).
\]
Other examples of composition complexes given in \cite{G} are the singular cochain complex of a topological space, and the cobar construction on a Hopf algebra.  
In general, one can work with any composition complex mimicking  constructions which are standard for  the Hochschild cochain complex, 
as the next definition shows.

\begin{definition}[\cite{G}]\label{cc_defs} We provide a brief dictionary between the context of this paper and that of composition complexes.
\begin{enumerate}
    \item\label{cc_1} 
    For $f\in C^m$ and $g\in C^n$ we define the circle product and the bracket
    \[
    f\circ g =\sum_{i=0}^{m-1} (-1)^{i(n-1)}f\circ_i g, \quad [f,g]=f\circ g-(-1)^{(m-1)(n-1)}g\circ f
    \]
    in $C^{m+n-1}$. These correspond to the usual circle product and Gerstenhaber bracket when $C=C^*(A)$.
    
    \item\label{cc_2} 
    We call $m=1\smile 1\in C^2$ the multiplication element of $C$---in the case of the Hochschild cochain complex this is  the multiplication map $A\otimes A \to A$. We then define a differential on $C$ by the rule $\partial(f)= [m,f]$, and $C$ becomes a complex $
    C^0\xrightarrow{\partial}C^1\xrightarrow{\partial}C^2\xrightarrow{\partial}\cdots$ 
    In particular we obtain cohomology groups ${\rm H}^i(C)$. When $C=C^*(A)$, these yield the usual Hochschild differential and Hochschild cohomology groups $\HH^i(A)$. 
    
    \item A derivation in $C$ is a degree one cycle $D\in \Der(A)=Z^1(C)=\ker(C^1\to C^2)$. An automorphism in $C$ is an element $\alpha\in C^1$, invertible with respect to the circle product, such that $\alpha\circ m = \alpha\smile\alpha$. When $C=C^*(A)$, these correspond to derivations and automorphisms of the $k$-algebra $A$.
    
    \item By base change, $C$ gives rise to a composition complex $C\llbracket t\rrbracket$ over the ring $k\llbracket t\rrbracket$. A one-parameter family of automorphisms in $C$ is an automorphism in $C\llbracket t\rrbracket$, and we write $\Aut_1(C\llbracket t\rrbracket)$ for the set of one-parameter families of automorphisms of the form $\alpha = 1 + \alpha_1t+\alpha_2t^2+\cdots $  A derivation $D\in \Der(C)$ is called integrable if $D=\alpha_1$ for some $\alpha\in \Aut_1(C\llbracket t\rrbracket)$. One can similarly define $[m,n)$-integrable derivations for any $m,n$ by considering automorphisms in $C[t]/(t^n)$. Once again, if $C=C^*(A)$ this yields the usual notion of integrable derivation. 
    \item Finally, if $\alpha\in \Aut_1(C[t]/(t^n))$, the obstruction theory of Subsection \ref{sub_obs} can be generalised by setting $\obs(\alpha)= [\alpha_1\smile \alpha_{n-1}+\cdots + \alpha_{n-1}\smile \alpha_1]\in {\rm H}^2(C)$.
    
\end{enumerate}
\end{definition}

With these definitions in place, the results stated in Section \ref{sec_int_Ger} all hold at the generality of any composition complex. Since Gerstenhaber was primarily concerned with automorphisms and derivations, which can be understood from the first few degrees, the results of \cite{G} are stated even more generally for composition complexes \emph{truncated in degree $2$}. That is, $k$-modules $C^0,C^1,C^2$ having the structure and properties of \cref{ccdef} to the extent that they are meaningful.

\begin{remark}
In modern terminology, a composition complex is the same thing as a nonsymmetric operad with multiplication \cite{GV}. For example, the composition complex $C^*(A)$ is the endomorphism operad of $A$. In \cite{GV} Gerstenhaber and Voronov explain how any nonsymmetric operad with multiplication inherits the structure of a B$_\infty$-algebra. This construction mirrors some of the ideas from Definition \ref{cc_defs}; in particular, they construct the bracket (\ref{cc_1}) and differential (\ref{cc_2}) exactly as was done originally in \cite{G}. Conversely there are many interesting examples of operads with multiplication (for example the Kontsevich operad used in \cite{GV}), and each can be considered as a composition complex, which thereby obtains a notion of integrable derivation and an obstruction theory as in definition \ref{cc_defs}.
\end{remark}


\begin{thebibliography}{}

\bibitem{Benson} J.\ Benson, 
{\em Representations and Cohomology, Vol. I: Cohomology of groups and modules},
Cambridge studies in advanced mathematics {\bf 30}, Cambridge University Press (1991).

\bibitem{BKL}
D. Benson, R.\ Kessar, M.\ Linckelmann, 
{\em Hochschild cohomology of symmetric groups in low degrees}, arXiv:2204.09970.

\bibitem{BR} 
B.\ Briggs, L.\ Rubio y Degrassi,
{\em Stable invariance of the restricted Lie algebra structure of Hochschild cohomology}, 
arXiv:2006.13871v2 (2020).

\bibitem{BR2} B.\ Briggs, L.\ Rubio y Degrassi
   {\em Maximal tori in $\HH^1$ and the fundamental group},
   Int.\ Math.\ Res.\ Not.


\bibitem{Cra}D.\ Craven,
{\em Representation theory of finite groups: a guidebook.} Universitext. Springer, Cham, 2019.

\bibitem{FGM}D.\ R.\ Farkas, C.\ Geiss, E.\ N.\ Marcos,  
{\em  Smooth automorphism group schemes.}  
Representations of algebras (S\~ao Paulo, 1999), 71--89, Lecture Notes in Pure and Appl. Math., {\bf 224} Dekker, New York (2002).

\bibitem{FJLM} P.\ Fleischmann, I.\ Janiszczak, W.\ Lempken, 
{\em Finite groups have local non-Schur centralizers.}
Manuscripta Math. {\bf 80} (1993), no. 2, 213--224.

\bibitem{G} M.\ Gerstenhaber,
{\em On the deformation of rings and algebras. III. }
Ann. of Math. (2) {\bf 88} (1968), 1--34. 

\bibitem{GV} M.\ Gerstenhaber,  A.\ A.\ Voronov, 
{\em Homotopy G-algebras and moduli space},
Int. Math. Res. Not. no. 3, 1995, 141--153, 

\bibitem{GJ} E.\ Getzler, J.\ D.\ S.\ Jones, 
{\em Operads, homotopy algebra, and iterated integrals for double loop spaces},
(hep-th/9403055).

\bibitem{HS} H.\ Hasse, F.\ K.\ Schmidt, 
{\em Noch eine Bergründung der Theorie der höheren Dif- ferentialquotienten
in einem algebraischen Funktionenkörper einer Unbestimmten}. J. Reine Angew.Math.
{\bf 177}, 215–237 (1937)

\bibitem{HZS}
B.\ Huisgen-Zimmermann, M.\ Saor\'in,
{\em Geometry of chain complexes and outer automorphisms
under derived equivalence.} 
Trans. Amer. Math. Soc. {\bf 353} (2001), no. 12, 4757--4777.

\bibitem{J} N.\ Jacobson, 
{\em Classes of restricted Lie algebras of characteristic $p$. II}. Duke Math. J. {\bf 10} (1943), 107--121

\bibitem{KellerI} B.\ Keller, 
{\em Derived invariance of higher structures on the Hochschild complex},
preprint, https://webusers.imj-prg.fr/~bernhard.keller/publ/dih.pdf, 2003.

\bibitem {KellerIII} B.\ Keller, 
{\em Hochschild cohomology and derived Picard groups},
J. Pure Appl. Algebra
{\bf 190}, Issues 1–3, (2004), 177--196.
 
\bibitem{L} M.\ Linckelmann 
{\em Finite-dimensional algebras arising as blocks
of finite group algebras,}  
Contemporary Mathematics {\bf 705}, (2018) 155--188.

\bibitem{ML}M.\ Linckelmann, {\em Integrable derivations and stable equivalences of Morita type}. Proc. Edinb. Math. Soc. (2) {\bf 61} (2018), no. 2, 343--362.

\bibitem{L1}M.\ Linckelmann, {\em The Block Theory of Finite Group Algebras, Volume 1}, LMS Student Society Texts 91, (2018).

\bibitem{L2}M.\ Linckelmann, 
{\em The Block Theory of Finite Group Algebras, Volume 2}, LMS Student Society Texts 92, (2018).

\bibitem{Ma} H.\ Matsumura, {\em Integrable derivations}, Nagoya Math. J. {\bf 87} (1982) 227--245.

\bibitem{Mo} S.\ Molinelli. {\em Sul modulo delle derivazioni integrabili in caractteristica positiva} .Ann. Mat. Pura Appl. {\bf 121} (1979), 25–38.

\bibitem{Rick} J.\ Rickard,
{\em Derived categories and stable equivalence.} J. Pure Appl. Algebra {\bf 61} (1989), no. 3, 303--317. 

\bibitem{Ri} J.\ Riordan,
{\em Combinatorial identities.} John Wiley \& Sons, Inc., New York-London-Sydney 1968 xiii+256 pp. 

\bibitem {Ruq}
R.\ Rouquier, 
{\em Groupes d’automorphismes et \'equivalences stables ou d\'eriv\'ees.} 
AMA-Algebra Montpellier Announcements-01-2003 (2003).

\bibitem{R}L.\ Rubio y Degrassi, 
{\em Invariance of the restricted p-power map on integrable derivations under stable equivalences}. J. Algebra {\bf 469} (2017), 288--301. 

\bibitem{R2}L.\ Rubio y Degrassi, {\em On Hochschild cohomology and modular representation theory}, (2016) PhD thesis. 

\bibitem{Skow} A.\ Skowro\'nski, {\em Self-injective algebras: finite and tame type}, in: Trends in Representation Theory of Algebras
and Related Topics, Contemporary Math. 406, Amer. Math. Soc., Providence, RI, (2006), 169–238.


\bibitem{Vo}P.\ Vojta. 
{\em Jets via Hasse–Schmidt derivations}. In “Diophantine geometry”, CRM Series, vol. 4, Ed. Norm., Pisa, 2007, 335–361.

\bibitem{Zim} A.\ Zimmermann, 
{\em Representation Theory. A Homological Algebra Point of View}, Springer, (2014).

\end{thebibliography}
\end{document}